\documentclass[10pt]{amsart}
\usepackage{amsthm}
\usepackage{amsmath}
\usepackage{bbm}
\usepackage{nicefrac}
\usepackage{graphicx}
\usepackage{amssymb}
\usepackage{color}
\usepackage{verbatim}
\usepackage[font=small,skip=5pt]{caption}
\usepackage{enumitem}
\usepackage[makeroom]{cancel}
\usepackage{mathtools}
\usepackage[most]{tcolorbox}
\usepackage[all]{xy}
\usepackage{tikz-cd}
\usepackage{hyperref}
\usepackage{bookmark}
\hypersetup{
	colorlinks=true,
	linkcolor=blue,
	filecolor=magenta,      
	urlcolor=blue,
    citecolor=blue,
}
\usepackage{mathrsfs}
\usepackage{todonotes}

\newtheorem{theorem}{Theorem}[section]
\newtheorem{lemma}[theorem]{Lemma}
\newtheorem{corollary}[theorem]{Corollary}
\newtheorem{proposition}[theorem]{Proposition}
\newtheorem{conjecture}[theorem]{Conjecture}
\newtheorem{problem}[theorem]{Problem}

\newtheorem{condition}{Condition}
\newtheorem{definition}[theorem]{Definition}
\makeatletter
\newtheorem*{innercustomgeneric}{\customgenericname}
\providecommand{\customgenericname}{}
\newcommand{\newreptheorem}[2]{%
  \newenvironment{rep#1}[1]{%
    \def\customgenericname{#2~\ref*{##1}}%
    \begin{innercustomgeneric}
  }{%
    \end{innercustomgeneric}
  }}
\makeatother

\newreptheorem{theorem}{Theorem}

\newcommand{\RR}{\mathbb{R}}
\newcommand{\CC}{\mathbb{C}}
\newcommand{\ZZ}{\mathbb{Z}}
\newcommand{\PP}{\mathbb{P}}

\newcommand{\Pic}[1]{\textrm{Pic}(#1)}

\newcommand{\sshf}[1]{\mathscr{O}_{#1}}

\newcommand{\iso}{\simeq}

\newcommand{\paren}[1]{\left(#1\right)}

\title{On Property $N_p$ of line bundles on smooth projective toric varieties}
\author{Lei Song and Huanqi Wen}

\begin{document}
\maketitle

\begin{abstract}
We establish a criterion for Property $N_p$ for line bundles on a class of smooth projective toric varieties. More precisely, we prove that if a smooth projective toric variety $X$ of dimension $n\ge2$ satisfies the uniform unimodularity condition and the Thomsen stratification intersection-number condition, then any line bundle $L$ on $X$ with $L\cdot C\ge n-1+p$ for every $T$-invariant curve $C$ satisfies Property $N_p$. We also show that these two conditions hold for several families of toric varieties and are preserved under finite products.
\end{abstract}


\section{Introduction}
Let $X$ be a projective variety and $L$ a basepoint-free line bundle on $X$. Property $N_p$ of $L$ can be defined as follows. Set $S=\text{Sym}^{\bullet}H^0(X,L)$ and $R(X,L)=\bigoplus_{m\ge 0} H^0(X, L^m)$ with the natural $S$-module structure. Let $E_{\bullet}$ be a minimal graded free resolution of $R(X,L)$. Then 
\begin{definition}
    $L$ satisfies Property $N_p$ for some integer $p\ge 0$ if
    \begin{enumerate}
        \item $E_0=S$, and
        \item $E_i=S(-i-1)^{\oplus b_i}$ for some $b_i\in \mathbb{N}$, for each $1\le i\le p$. 
    \end{enumerate}
\end{definition}
Property $N_p$ measures the complexity of the minimal free resolution of the section ring $R(X,L)$. In particular, in case $L$ is very ample, $N_0$ corresponds to projective normality, while $N_1$ further requires that the defining ideal of $X$ be generated by quadrics. Higher Properties $N_p$ imposes linearity conditions on higher syzygies. 

A central problem is to find effective positivity conditions ensuring that a given line bundle satisfies Property $N_p$. For example, Green proved that on a smooth projective curve of genus $g$, a line bundle of degree at least $2g+1+p$ satisfies $N_p$ \cite{Green1984}, extending the classical result of Castelnuovo, Mattuck and Mumford on normal generation. In higher dimensions, Mukai conjectured that for a smooth projective variety $X$ of dimension $n$ and an ample divisor $A$, the adjoint line bundle $\sshf{X}(K_X + rA)$ satisfies $N_p$ whenever $r \geq n+2+p$. This conjecture has been verified in several important cases; for example, \cite{EinLazarsfeld1993arbitrarydim, BangereLacini2025}, but is still widely open. 

Turning to projective toric varieties, the study of Property $N_p$ has also attracted considerable attention. Every $n$-dimensional lattice polytope $P$ determines an $n$-dimensional projective toric variety $X_P$ together with a torus-invariant ample divisor $D_P$. We say that $P$ satisfies Property $N_p$ if the associated line bundle $\sshf{X_P}(D_P)$ does. In dimension two, Koelman \cite{Koelman1993, Koelman1993Generators} showed that every lattice polygon satisfies $N_0$, and provided criteria for $N_1$ in terms of boundary lattice points. Schenck \cite{Schenck2004} further proved that if a polygon contains at least $p+3$ lattice points, then it satisfies $N_p$. In higher dimensions, one research  focus has been to consider powers $L=A^{\otimes d}$ of a given ample line bundle $A$. Ogata and Nakagawa \cite{OgataNakagawa2002} proved that $L=A^{\otimes d}$ satisfies $N_0$ for $d \geq n-1$ and $N_1$ for $d \geq n$, and this was extended by Hering, Schenck, and Smith \cite{HeringSchenckSmith2006} to any  $N_p$ for $d \geq n-1+p$. Note that the polytope associated with $(X,A^{\otimes d})$ is the $d$-fold dilation of the polytope associated with $(X,A)$. For general lattice polytopes, Gubeladze \cite{Gubeladze2012} established $N_0$ in terms of
the lattice length $L(P)$ (see \S 3 for definition). Specifically, he showed that $P$ satisfies $N_0$ if $L(P)\ge 4n(n+1)$ and that this bound can
be relaxed to $L(P)\ge n(n+1)$ when $P$ is a  simplex.

Motivated by Mukai's conjecture, the first author and Zhixian Zhu proposed the following conjecture; see \cite[\S 5]{SWZ26}: 
\begin{conjecture}[Song-Zhu]\label{conj:Song-Zhu}
    Let $X$ be a projective toric variety of dimension $n\ge 2$ and $L$ a line bundle on $X$. If $L \cdot C \geq n-1+p$ for every $T$-invariant curve $C \subset X$, then $L$ satisfies Property $N_p$.
\end{conjecture}
We note that, by \cite{GZ22}, if $L$ satisfies $L\cdot C\ge n-1$ for every $T$-invariant curve $C$, then $L$ is very ample.  
When $p=0$, the conjecture can be reformulated as follows: Let $P$ be an $n$-dimensional lattice polytope with lattice length $L(P) \geq n-1$, then for any integer $r\ge 1$, every lattice point in $rP$ can be decomposed to a sum of $r$ lattice points in $P$. This decomposition property has been established when $P$ is a simplex \cite{SWZ26, SWZ24}.

While Conjecture \ref{conj:Song-Zhu} is supported by the evidence discussed above, no general result has yet been proven to hold for arbitrary $n, p$ and line bundles $L$ satisfying the positivity constraint. In this article, we verify Conjecture \ref{conj:Song-Zhu} under additional hypotheses on the toric variety. More specifically, we prove the following result:

\begin{theorem}\label{thm: main theorem}
    Let $X$ be a smooth projective toric variety of dimension $n\ge 2$ satisfying the uniform unimodularity condition and the Thomsen stratification intersection-number condition. 
    Let $p$ be a non-negative integer. Given a line bundle $L$ on $X$ such that $L\cdot C\ge n-1+p$ for every $T$-invariant curve $C$ on $X$, then $L$ satisfies Property $N_p$.
\end{theorem}

Our approach combines cohomological methods with the combinatorial structure of toric varieties. Using Green’s criterion, we reduce the problem to the vanishing of certain cohomology groups of ideal sheaves of the diagonals. To make these vanishings tractable, we employ the Hanlon-Hicks-Lazarev resolution \cite{HanlonHicksLazarev2024} together with the Thomsen collection, which allows us to control the cohomology via intersection numbers. 

In Proposition \ref{prop: example}, we show that three families of toric varieties satisfy the additional conditions: the blow-up of $\PP^n$ at $k$ points in general position, for $0\le k\le n+1$; the blow-up of $\PP^n$ along $\PP^k$, for $0\le k\le n-1$; and projective bundles $\mathbb{P}(\mathcal{O}_{\PP^s}^{\oplus r_1}\oplus \mathcal{O}_{\PP^s}(1)^{\oplus r_2})$, for $s, r_1, r_2\ge 1$. Moreover, we show that these two conditions are preserved under taking finite products, see Propositions \ref{prop: product with cond 1}, \ref{prop:product with cond 2}.
\vspace{1mm}

\textbf{Acknowledgments}.
We are grateful to Zhixian Zhu for the collaboration on related projects and for many helpful discussions. We also thank Lawrence Ein, Lauren Cranton Heller and Jeff Hicks for helpful discussions and communications. 
During the preparation of the article, L.S. was partially supported by NSFC grant No.~12471043 and  Guangdong Basic and Applied Basic Research Foundation No.~ 2025A1515012258. 

\section{Hypotheses for the Main Theorem}
In this section, we first introduce the Thomsen collection and the Thomsen stratification. We then present the two conditions appearing in our main theorem, and explain their motivation as well as their relation to the Thomsen collection and the Thomsen stratification. Finally, we provide several examples of toric varieties satisfying these conditions. 

To begin with, we fix some notation, and refer the reader to \cite{FultonIntroToToric, CoxToric}. Let $N\cong \ZZ^n$ be a lattice of rank $n$ and $M=\mathrm{Hom}(N,\ZZ)$ be the dual lattice. We write $\langle . , . \rangle: M\times N\rightarrow \mathbb{Z}$ for the natural pairing. Let $\Sigma$ be a fan in $N_\RR=N\otimes_\ZZ \RR$ and $X=X_\Sigma$ be the associated toric variety. Then $X$ is a normal variety containing a dense open torus $T=\mathrm{Spec}(\CC[M])\cong (\CC^*)^n$, and the natural action of $T$ on itself extends to an action of $T$ on $X$. For integer $k\ge 0$, $\Sigma(k)$ denotes the set of $k$-dimensional cones. For each ray $\rho\in \Sigma(1)$, we let $u_\rho\in N$ denote its primitive generator.

Throughout, we work over an algebraically closed field, and we always assume that $X$ is a smooth toric variety, unless otherwise specified.

\subsection{Thomsen collection and Thomsen stratification}
For a smooth toric variety $X$, one can construct a collection of line bundles on $X$, called the Thomsen collection. 

\begin{definition}
	For each $\theta\in M_{\RR}$, define
	\[d(\theta):=\sum_{\rho\in \Sigma(1)}\lceil\langle \theta, u_{\rho}\rangle\rceil D_{\rho}.\]
	It follows that
    \begin{equation}
        \label{eq: -d(theta)}
        -d(\theta)=\sum_{\rho\in \Sigma(1)}\lfloor\langle -\theta, u_{\rho}\rangle\rfloor D_{\rho}.
    \end{equation}
	Thus we obtain a map
	\[
		\begin{aligned}
			-d: M_{\RR}/M & \to \Pic{X}\\
			[\theta] & \mapsto \mathcal{O}_X(-d(\theta)),
		\end{aligned}
	\]
	where $[\theta]$ denotes the class of $\theta$ in $M_{\RR}/M$. The image of $-d$ is called the Thomsen collection, denoted by $\Theta_X$.
\end{definition}
% By abuse of notation, we also write $d(\theta)$ for the line bundle $\mathcal{O}_X(d(\theta))$, and $-d(\theta)$ for $\mathcal{O}_X(-d(\theta))$.

In \cite{ballard2025kingsconjecturecoxcategory}, the authors prove the following vanishing property of the elements of $\Theta_X$.

\begin{lemma}
    \label{le: cohomology vanishing for Thomsen collection}
Assume that $|\Sigma|$ is convex and that $\Sigma$ is simplicial. Let $A=\sum_{\rho\in \Sigma(1)} a_\rho D_\rho$ be a nef divisor on $X=X_{\Sigma}$. Then for any $\mathcal{O}_X(-d(\theta))\in\Theta_X$ and any integer $p>0$, one has
\[
H^p(X,\mathcal{O}_X(A-d(\theta)))=0.
\]
\end{lemma}

The Thomsen stratification is a stratification of the real torus $T_M:=M_{\RR}/M$, which was introduced by Bondal in \cite{Bondal06}. More precisely, for each $\rho\in \Sigma(1)$, we define
\begin{equation}
\label{eq: chi_rho}
	\begin{aligned}
		\chi_{\rho}: T_M & \to \RR/\ZZ\\
		[m]& \mapsto \langle m, u_{\rho}\rangle \bmod \ZZ, 
	\end{aligned}
\end{equation}
and set $T_M(\rho):=\ker(\chi_{\rho})$. Then $T_M(\rho)$ is a hyperplane of $T_M$. The set of toric hyperplanes $\{T_M(\rho)\mid \rho\in \Sigma(1)\}$ induces a stratification of $T_M$, denoted by $S_X$, called the \textit{Thomsen stratification}.

For each stratum $\sigma$ in the Thomsen stratification, choose $\theta\in \sigma$. 
We then associate to $\sigma$ the line bundle 
\begin{equation}
    \label{eq: def of d_sigma}
    \mathcal{O}_X(-d_{\sigma}):=\mathcal{O}_X(-d(\theta)),
\end{equation}
which is an element of the Thomsen collection. It is easy to check that $\mathcal{O}_X(-d_{\sigma})$ is independent of the choice of $\theta\in \sigma$. By abuse of notation, we also write $-d_{\sigma}$ and $d_\sigma$ for a divisor representing the associated line bundle.

%Then the shape of the HHL resolution for diagonal toric varieties is determined by the Thomsen stratification

\subsection{Uniform unimodularity condition}
\begin{condition}[Uniform unimodularity condition]
	We say that $X$ satisfies the uniform unimodularity condition if for any $n$ linearly independent one-dimensional cones $\rho_1,\dots,\rho_n$ in the fan $\Sigma$ (where linearly independent means that their primitive generators are linearly independent in $N_{\mathbb{R}}$), the primitive generators $u_{\rho_1},\dots,u_{\rho_n}$ form a basis of $N$, that is, $\det(u_{\rho_1},\ldots,u_{\rho_n})=\pm 1$. 
\end{condition}

%This condiition is closely related to shape of HHL resolution for diagonal toric varieties. To be more precise,
Note that the uniform unimodularity condition is strictly stronger than smoothness for toric varieties. The following proposition shows the relationship between the uniform unimodularity condition and the Thomsen stratification. 
\begin{proposition}
	\label{prop: equivalent condition for trivial 0 stratum}
	Let $X$ be a smooth toric variety. The following are equivalent:
	\begin{enumerate}
		\item $X$ satisfies the uniform unimodularity condition;
		\item the Thomsen stratification $S_X$ has only one 0-dimensional stratum, which equals $\{0\}$.
	\end{enumerate}
\end{proposition}

Before proving this proposition, we give a lemma. 
First, for each stratum $\sigma$ in $S_X$, it is given by a connected component of 
\[
R_J:=\bigcap_{\rho\in J} T_M(\rho)\setminus \bigcup_{\rho\in\Sigma(1)\setminus J} T_M(\rho),
\]
where $J\subseteq \Sigma(1)$. Then define the sub-lattice of $N$
\[
N_J=\sum_{\rho\in J} \mathbb{Z} u_{\rho}.
\]

\begin{lemma}
    \label{prop: dimension of strata}
    Let $\sigma\neq \emptyset$ be a stratum given by a connected component of $R_J$. Then
    \begin{enumerate}
        \item if $\rho\in \Sigma(1)\setminus J$, then $u_{\rho}\notin N_J$;
        \item There is a canonical isomorphism of topological groups 
        \[G:=\bigcap_{\rho\in J} T_M(\rho)\iso \operatorname{Hom}_{\mathbb{Z}}(N/N_J,\mathbb{R}/\mathbb{Z});\]
        \item $\dim \sigma=\operatorname{rank}(N)-\operatorname{rank}(N_J)$. In particular, every connected component of $R_J$ has the  same dimension.
    \end{enumerate}
\end{lemma}

\begin{proof}
    (1) We shall argue by contradiction. Suppose that there exists $\rho_0\in \Sigma(1)\setminus J$ such that $u_{\rho_0}\in N_J$. Then
    \[
    u_{\rho_0}=\sum_{\rho\in J} z_{\rho} u_{\rho},
    \]
    for some $z_{\rho}\in \mathbb{Z}$. Hence for any $m\in M_{\mathbb{R}}$, we have
    \[
    \langle m, u_{\rho_0}\rangle =\sum_{\rho\in J} z_{\rho} \langle m, u_{\rho}\rangle ,
    \]
    and therefore
    \[
    T_M(\rho_0)\supseteq \bigcap_{\rho\in J} T_M(\rho).
    \]
    This implies that $R_J=\emptyset$, contradicting $\sigma\neq \emptyset$.

    (2)
    We first claim there is a canonical isomorphism 
\[
\Phi: T_M=M_{\mathbb R}/M
\cong \operatorname{Hom}_{\mathbb Z}(N,\mathbb R/\mathbb Z)
\]
as topological groups.
Indeed, the quotient map $\mathbb R\to \mathbb R/\mathbb Z$ induces a continuous homomorphism
\[
\operatorname{Hom}_{\mathbb Z}(N,\mathbb R)
\longrightarrow
\operatorname{Hom}_{\mathbb Z}(N,\mathbb R/\mathbb Z).
\]
Since $N$ is a free abelian group of finite rank, this map is surjective, and its kernel is
\[
\operatorname{Hom}_{\mathbb Z}(N,\mathbb Z)=M.
\]
Hence it induces a continuous bijective homomorphism
\[
M_{\mathbb R}/M
\longrightarrow
\operatorname{Hom}_{\mathbb Z}(N,\mathbb R/\mathbb Z).
\]
Both sides are compact Hausdorff topological groups, so this map is a homeomorphism. 

Under $\Phi$, the subgroup 
\[
G=\{x\in T_M\mid \langle x,u_\rho\rangle\in \mathbb Z
\text{ for all } \rho\in J\}
\]
of $T_M$ is isomorphic to the subgroup 
\[
\left\{
\varphi\in \operatorname{Hom}_{\mathbb Z}(N,\mathbb R/\mathbb Z)
\ \middle|
\varphi(N_J)=0
\right\}
\]
of $\operatorname{Hom}_{\mathbb Z}(N,\mathbb R/\mathbb Z)$ as topological groups.

The quotient map $q: N\rightarrow N/{N_J}$ induces a continuous injective homomorphism
\[q^*: \operatorname{Hom}_{\mathbb Z}(N/{N_J}, \mathbb R/{\mathbb Z})\rightarrow \operatorname{Hom}_{\mathbb Z}(N,\mathbb R/\mathbb Z)\]
with image 
\[
\text{im}(q^*)=\left\{
\varphi\in \operatorname{Hom}_{\mathbb Z}(N,\mathbb R/\mathbb Z)
\ \middle|
\varphi(N_J)=0
\right\}.
\]

It remains to show that $q^*$ is a homeomorphism onto its image. To this end, note that $N/N_J$ is finitely generated, we can write
\[
N/{N_J}\cong \mathbb Z^{\oplus r}\oplus T\]
for some integer $r\ge 0$  and finite abelian group $T$. Hence
\[
\operatorname{Hom}_{\mathbb Z}(N/{N_J}, \mathbb R/{\mathbb Z})
\cong
\prod^r (\mathbb R/\mathbb Z)\times \widehat T,\]
where $\widehat T$ is finite. This implies that
$\operatorname{Hom}_{\mathbb Z}(N/N_J,\mathbb R/\mathbb Z)
$
is compact. On the other  hand, the target
$\operatorname{Hom}_{\mathbb Z}(N,\mathbb R/\mathbb Z)\cong \prod^n \mathbb (R/\mathbb Z)$ is Hausdorff, therefore $q^*$ is a homeomorphism onto its image.

Combining this with the homeomorphism $\Phi$, we obtain a canonical isomorphism from $G$ to $\operatorname{Hom}_{\mathbb{Z}}(N/N_J,\mathbb{R}/\mathbb{Z})$, as asserted.

(3) In view of (2), we have
\[\dim(G)=\operatorname{rank}(N/N_J)=\operatorname{rank}(N)-\operatorname{rank}(N_J)=r.
    \]
Take any $\rho'\notin J$ and let $\bar{u}_{\rho'}$ denote the image of $u_{\rho'}$ in $N/N_J$. %Then $N_{J'}/N_J$ is the subgroup of $N/N_J$ generated by $\bar{u}_{\rho'}$. 
    We claim that every connected component of $G\setminus T_M(\rho')$ has dimension $r$. Indeed, there are two cases.
    \begin{itemize}
        \item If $\bar{u}_{\rho'}$ has finite order in $N/N_J$, then $\bar{u}_{\rho'}$ lies entirely in the torsion part of $N/N_J\cong \mathbb{Z}^r\oplus T$, the condition $\chi_{\rho'}(x)=0$ for $x\in T_M$ imposes no condition on the $(\mathbb{R}/\mathbb{Z})^r$-factor and cuts out a proper subgroup of $\widehat{T}$, where $\chi_{\rho'}$ is given by \eqref{eq: chi_rho}. Hence $G\cap T_M(\rho')$ is a union of connected components of $G$, so $G\setminus T_M(\rho')$ is also a union of connected components of $G$. Therefore every connected component of $G\setminus T_M(\rho')$ has dimension $r$.
        \item If $\bar{u}_{\rho'}$ has infinite order in $N/N_J$, that is, if $\bar{u}_{\rho'}$ has a nonzero component in the $\mathbb{Z}^r$-part of $N/N_J\cong \mathbb{Z}^r\oplus T$, then $G\cap T_M(\rho')$ on each connected component of $G$ is either empty or $(r-1)$-dimensional subtorus of that component. It follows that every connected component of $G\setminus T_M(\rho')$ again has dimension $r$.
    \end{itemize}

    Note that
    \[
    R_J=G\setminus \bigcup_{\rho\in\Sigma(1)\setminus J} T_M(\rho).
    \]
    Starting from $G$ and removing all $T_M(\rho)$ for $\rho\in \Sigma(1)\setminus J$ one by one, it follows inductively that every nonempty connected component of $R_J$
    has dimension $r$. Since $\sigma$ is a connected component of $R_J$, we conclude that
    \[
    \dim \sigma=\dim G=\operatorname{rank}(N)-\operatorname{rank}(N_J).
    \]
\end{proof}

\begin{proof}[Proof of Proposition \ref{prop: equivalent condition for trivial 0 stratum}]
	$(1)\Longrightarrow(2)$:
Assume that for any $n$ linearly independent one-dimensional cones in $\Sigma$, their primitive generators form a basis of $N$. Let $\sigma$ be a $0$-dimensional stratum. Since every stratum in $S_X$ is given by a connected component of
\[
R_J=\bigcap_{\rho\in J} T_M(\rho)\setminus \bigcup_{\rho\in\Sigma(1)\setminus J} T_M(\rho),
\]
where $J\subseteq \Sigma(1)$, it follows that $\sigma$ is a connected component of $R_J$ for some $J$.

Since $\dim(\sigma)=0$, Lemma~\ref{prop: dimension of strata} implies that
\[
\operatorname{rank}(N_J)=\operatorname{rank}(N)=n.
\]
Therefore, we can choose $n$ one-dimensional cones $\rho_1,\dots,\rho_n$ in $J$ such that $u_{\rho_1},\dots,u_{\rho_n}$ are linearly independent in $N_{\mathbb{R}}$. By assumption, these vectors form a basis of $N$. Hence
\[
\bigcap_{i=1}^n T_M(\rho_i)=\{0\}.
\]
On the other hand,
\[
\sigma\subseteq R_J\subseteq \bigcap_{i=1}^n T_M(\rho_i).
\]
Therefore, $\sigma=R_J=\{0\}$, that is, $S_X$ has only one 0-dimensional stratum, which equals $\{0\}$

$(2)\Longrightarrow(1)$:
We argue by contradiction. Suppose that there exist $n$ linearly independent one-dimensional cones in $\Sigma$ such that their primitive generators $u_{\rho_1},\dots,u_{\rho_n}$ do not form a basis of $N$. Let
\[
L=\sum^n_{i=1} \mathbb{Z}u_{\rho_i}.
\]
Then $[N:L]>1$. Moreover,
\[
\bigcap_{i=1}^n T_M(\rho_i)
=
\{x\in T_M\mid \langle x,u_{\rho_i}\rangle\in\mathbb{Z},\ i=1,\dots,n\}
\cong \operatorname{Hom}(N/L,\mathbb{R}/\mathbb{Z}).
\]
Since $N/L$ is finite and $[N:L]>1$, the set $\bigcap_{i=1}^n T_M(\rho_i)$ is $0$-dimensional and contains a nontrivial point. Choose a nontrivial point
\[
x\in \bigcap_{i=1}^n T_M(\rho_i),
\]
and let $J=\{\rho_1,\dots,\rho_n\}$. Now take any $\rho\in\Sigma(1)\setminus \{\rho_1,\dots,\rho_n\}$. If $x\in T_M(\rho)$, add $\rho$ to $J$; otherwise leave $J$ unchanged. Repeating this process for all $\rho\in \Sigma(1)\setminus \{\rho_1,\dots,\rho_n\}$, we obtain a new index set, still denoted by $J$, such that $x\in R_J$. Thus we obtain a nontrivial $0$-dimensional stratum, contradicting the assumption. This proves the claim.
\end{proof}

\begin{proposition}
    \label{prop: product with cond 1}
    If smooth toric varieties $X_1,\cdots,X_m$ satisfy the uniform unimodularity condition, then their product $X=\prod_{i=1}^m X_i$ also satisfies this condition.
\end{proposition}

\begin{proof}
    Since $X=\prod_{i=1}^m X_i$, we have
    \[
    N=\bigoplus_{i=1}^m N_i,\quad
    M=\bigoplus_{i=1}^m M_i,\quad
    \Sigma_X=\Sigma_{X_1}\times\cdots\times\Sigma_{X_m}, \text{ and }
    T_M=\prod_{i=1}^m T_{M_i}.
    \]
    Every one-dimensional cone in $\Sigma_X$ is of the form
    \[
    \widetilde{\rho}
    =
    \{0\}\times\cdots\times \rho \times\cdots\times \{0\},
    \]
    where $\rho\in \Sigma_{X_i}(1)$. Thus the primitive generator of $\widetilde{\rho}$ is
    \[
    u_{\widetilde{\rho}}=(0,\dots,0,u_{\rho},0,\dots,0)\in N.
    \]
    Then one can check that the primitive generators in the fan $\Sigma_X$ satisfy uniform unimodularity condition provided that $\Sigma_{X_i}$ satisfies this condition for each $i$.
\end{proof}

\subsection{Thomsen stratification intersection-number condition}
We define two invariants for smooth toric varieties:
\[
B(X):=\max_{\substack{\text{$C$ is a $T$-invariant curve on $X$}\\ \text{$\sigma$ is a $0$-dimensional stratum in $S_X$}}} d_{\sigma}\cdot C,
\]
\[
A(X):=\max_{\substack{\text{$C$ is a $T$-invariant curve on $X$}\\ \rho\in\Sigma(1)}} D_{\rho}\cdot C.
\]
Note that $A(X)\ge 1$ for every smooth toric variety.
\begin{condition}[Thomsen stratification intersection-number condition]
    \label{cond: intersection number of Thomsen collection}
    Let $S_X$ be the stratification of $T_M$ induced by the smooth toric variety $X$. Then for any $k$-dimensional stratum $\sigma\in S_X$ and any $T$-invariant curve $C$ on $X$, we have
    \[
    d_{\sigma}\cdot C\le B(X)+kA(X).
    \]
\end{condition}

When $S_X$ has only one $0$-dimensional stratum, the invariants $B(X)$ and $A(X)$ take simple values.

\begin{lemma}
    \label{le: A(X)=1}
    Let $X$ be a smooth toric variety of dimension $n$. If $S_X$ has only one $0$-dimensional stratum, then $B(X)=0$ and $A(X)=1$.
\end{lemma}

\begin{proof}
    It is clear that $B(X)=0$. Let $C$ be a $T$-invariant curve given by an $(n-1)$-dimensional cone $\tau$, that is, $C=V(\tau)$. Assume that $\tau$ is a common face of two $n$-dimensional cones $\sigma$ and $\sigma'$. Since $X$ is smooth, we may assume that the primitive generators of $\tau$ are $u_1,\dots,u_{n-1}$, while the primitive generators of $\sigma$ and $\sigma'$ are obtained by adding $u_n$ and $u_n'$ into $u_1,\dots,u_{n-1}$ respectively. Moreover, both sets
    \[
    u_n,u_1,\dots,u_{n-1}
    \quad\text{and}\quad
    u_n',u_1,\dots,u_{n-1}
    \]
    form bases of $N$. By \cite[\S 5.1]{FultonIntroToToric}, there exist unique integers $a_1,\dots,a_{n-1}$ such that
    \[
    u_n+u_n'+\sum_{i=1}^{n-1} a_i u_i=0.
    \]
    Moreover, we have
    \begin{equation}
        \label{eq:intersection number}
        D_{\rho}\cdot C=
        \begin{cases}
            1 & \text{if $\rho$ is generated by $u_n$ or $u_n'$,}\\
            a_i & \text{if $\rho$ is generated by $u_i, 1\le i\le n-1$,}\\
            0 & \text{for all other $\rho\in\Sigma(1)$.}
        \end{cases}
    \end{equation}

    By assumption, $S_X$ has only one $0$-dimensional stratum. It follows from Proposition~\ref{prop: equivalent condition for trivial 0 stratum} that for every $1\le i\le n-1$, $u_n',u_1,\cdots, \hat{u_i},\cdots,u_n$ are either linearly dependent or form a basis of $N$. Hence, for every $1\le i\le n-1$, $a_i\in\{0,1,-1\}$. Then by \eqref{eq:intersection number}, for any $T$-invariant curve $C$ and any $\rho\in\Sigma(1)$, we have
    \[
    D_{\rho}\cdot C\le 1.
    \]
    On the other hand, there exist $\rho$ and $C$ such that $D_{\rho}\cdot C=1$. Therefore $A(X)=1$.
\end{proof}

Thus, if $X$ satisfies the uniform unimodularity condition, we can write the Thomsen stratification intersection-number condition as follows: for any $k$-dimensional stratum $\sigma\in S_X$ and any $T$-invariant curve $C$ on $X$, one has
\begin{equation}
    \label{eq: simplified intersection number condition}
    d_{\sigma}\cdot C \le k.
\end{equation}

\begin{proposition}\label{prop:product with cond 2}
    Let $X=\prod_{i=1}^m X_i$, where each $X_i$ is a smooth toric variety. Assume that each $X_i$ satisfies the Thomsen stratification intersection-number condition, that is, for every $k_i$-dimensional stratum $\sigma_i\in S_{X_i}$ and every $T_i$-invariant curve $C_i\subset X_i$, one has
    \[
    d_{\sigma_i}\cdot C_i\le B(X_i)+k_iA(X_i).
    \]
    Then for every $k$-dimensional stratum $\sigma\in S_X$ and every $T$-invariant curve $C\subset X$, we have
    \[
    d_\sigma\cdot C\le B(X)+kA(X).
    \]
\end{proposition}

\begin{proof}
For $X$ and its factors $X_i$, we have
\[
N=\bigoplus_{i=1}^m N_i,
\quad
M=\bigoplus_{i=1}^m M_i,
\quad
\Sigma=\Sigma_1\times \cdots \times \Sigma_m,\text{ and }
T_M=\prod_{i=1}^m T_{M_i}.
\]
Moreover, every stratum in $S_X$ is of the form
\[
\sigma=\sigma_1\times \cdots \times \sigma_m,
\quad
\sigma_i\in S_{X_i},
\]
and satisfies
\[
\dim(\sigma)=\sum_{i=1}^m \dim(\sigma_i).
\]

\textbf{Step 1: Decompose $d_\sigma$ as a sum of pullbacks of the $d_{\sigma_i}$.}
Let $\pi_i:X\to X_i$ be the projection map. Every one-dimensional cone $\widetilde{\rho}\in \Sigma(1)$ comes from some $\rho\in \Sigma_i(1)$, and the corresponding divisor satisfies
\begin{equation}
\label{eq: D_tilde_rho}
    D_{\widetilde{\rho}}=\pi_i^*D_\rho.
\end{equation}
Write $\sigma=\sigma_1\times \cdots \times \sigma_m$, and choose $\theta=(\theta_1,\dots,\theta_m)\in \sigma$. Then
\begin{equation}
    \label{eq: decomposition of d_sigma}
    \begin{aligned}
        d_\sigma
        &=
        \sum_{\widetilde{\rho}\in \Sigma(1)}
        \left\lceil \langle \theta,u_{\widetilde{\rho}}\rangle \right\rceil
        D_{\widetilde{\rho}} \\
        &=
        \sum_{i=1}^m
        \sum_{\rho\in \Sigma_i(1)}
        \left\lceil \langle \theta_i,u_\rho\rangle \right\rceil
        \pi_i^*D_\rho \\
        &=
        \sum_{i=1}^m \pi_i^*(d_{\sigma_i}).
    \end{aligned}
\end{equation}

\textbf{Step 2: Decomposition of $T$-invariant curves on $X$.}
Every cone in $\Sigma$ is of the form
\[
\tau=\tau_1\times \cdots \times \tau_m
\quad
(\tau_i\in \Sigma_i).
\]
Since every $T$-invariant curve on $X$ corresponds to a cone $\tau$ of codimension $1$, there exists a unique index $j$ such that $\tau$ has the form
\[
\tau=\tau_1\times \cdots \times \tau_m,
\]
where
\[
\operatorname{codim}(\tau_j)=1,
\quad
\operatorname{codim}(\tau_i)=0 \ \text{for all } i\neq j.
\]
Equivalently,
\begin{equation}
    \label{eq: invariant curve on product}
    C=
    \{x_1\}\times \cdots \times \{x_{j-1}\}\times C_j\times
    \{x_{j+1}\}\times \cdots \times \{x_m\},
\end{equation}
where $x_i$ is the torus-fixed point determined by $\tau_i$ for $i\neq j$, and $C_j\subset X_j$ is the invariant curve determined by $\tau_j$.

For such a curve, if $D_i$ is a divisor on $X_i$, then
\begin{equation}
    \label{eq: intersection}
    \pi_i^*D_i\cdot C=
    \begin{cases}
    D_j\cdot C_j,& i=j,\\
    0,& i\neq j.
    \end{cases}
\end{equation}

\textbf{Step 3: Compute $A(X)$ and $B(X)$.}
We first claim that
\begin{equation}
    \label{eq: max A_i}
    A(X)=\max_{1\le i\le m} A(X_i).
\end{equation}

Indeed, by \eqref{eq: D_tilde_rho} and \eqref{eq: intersection}, it follows immediately that
\[
A(X)\le \max_i A(X_i).
\]
Conversely, choose $\rho_i\in \Sigma_i(1)$ and a curve $C_i\subset X_i$ such that
\[
D_{\rho_i}\cdot C_i=A(X_i).
\]
Then choose a torus-fixed point on each $X_j$ for $j\neq i$, and take the product with $C_i$ to obtain a curve $C$ on $X$. We then have
\[
\pi_i^*D_{\rho_i}\cdot C=D_{\rho_i}\cdot C_i=A(X_i),
\]
and hence
\[
A(X)\ge A(X_i).
\]
Since $i$ is arbitrary, we obtain
\[
A(X)\ge \max_i A(X_i).
\]

Similarly, using the expression \eqref{eq: decomposition of d_sigma} for $d_\sigma$, the description \eqref{eq: invariant curve on product} of $T$-invariant curves, and the intersection formula \eqref{eq: intersection}, one proves in the same way that
\begin{equation}
    \label{eq:max B_i}
    B(X)=\max_{1\le i\le m} B(X_i).
\end{equation}

\textbf{Step 4: Proof of $d_\sigma\cdot C\le B(X)+kA(X)$.}
Take a $k$-dimensional stratum $\sigma\in S_X$ and a $T$-invariant curve $C\subset X$. By Steps 1 and 2, we may write
\[
\sigma=\sigma_1\times \cdots \times \sigma_m,
\quad
C=\{x_1\}\times \cdots \times C_j\times \cdots \times \{x_m\}.
\]
By \eqref{eq: decomposition of d_sigma} and \eqref{eq: intersection}, we obtain
\begin{equation}
    d_\sigma\cdot C=d_{\sigma_j}\cdot C_j.
\end{equation}
Let
\[
k_i=\dim(\sigma_i),
\quad
k=\sum_{i=1}^m k_i.
\]
By the assumption of the proposition,
\[
d_\sigma\cdot C
=
d_{\sigma_j}\cdot C_j
\le
B(X_j)+k_jA(X_j).
\]
Combining this with \eqref{eq: max A_i} and \eqref{eq:max B_i}, we obtain
\[
d_\sigma\cdot C
\le
B(X)+kA(X).
\]
This completes the proof.
\end{proof}

\subsection{Examples}
In this subsection, we verify the two conditions introduced above for three basic families of smooth toric varieties:
\begin{itemize}
    \item For $0 \le k \le n+1$, let $p_0, p_1, \dots, p_{k-1}$ be $k$ points in general position in $\mathbb{P}^n$, and define
    \[
    A_{n,k} := \operatorname{Bl}_{\{p_0,p_1,\dots,p_{k-1}\}}(\mathbb{P}^n).
    \]

    \item For $n-k \ge 2$, let $Z \subseteq \mathbb{P}^n$ be a $k$-dimensional linear subspace, and define
    \[
    B_{n,k} := \operatorname{Bl}_Z(\mathbb{P}^n).
    \]

    \item For $s,r_1,r_2 \ge 1$, define
    \[
    C_{s,r_1,r_2} := \mathbb{P}\bigl(\mathcal{O}_{\mathbb{P}^s}^{\oplus r_1}\oplus \mathcal{O}_{\mathbb{P}^s}(1)^{\oplus r_2}\bigr).
    \]
\end{itemize}

\begin{proposition}\label{prop: example}
    
    The varieties $A_{n,k}$, $B_{n,k}$, and $C_{s,r_1,r_2}$ satisfy the uniform unimodularity condition and the Thomsen stratification intersection-number condition.
\end{proposition}
\begin{proof}
 By applying a $\operatorname{PGL}(n+1)$-transformation to $\mathbb{P}^n$, we may assume that
\[
p_0=[1:0:\cdots:0],\ \cdots,\ p_{k-1}=[0:\cdots:0:1:0:\cdots:0]
\]
for $A_{n,k}$. We may also assume that
\[
Z=\{[x_0:\cdots:x_n]\in\mathbb{P}^n\mid x_{k+1}=\cdots=x_n=0\}
\]
for $B_{n,k}$.

For $A_{n,k}$, the fan $\Sigma_{A_{n,k}}$ is in $N_{\RR}\cong \RR^n$. Let $e_1,\cdots,e_n$ be a basis of $N$. The one-dimensional cones of $\Sigma_{A_{n,k}}$ are generated by
\[
-u:=\sum_{i=1}^n e_i, e_1,\cdots, e_n, u, -e_1,\cdots, -e_{k-1}.\]
If $k=0$, then the one-dimensional cones of $\Sigma_{\mathbb{P}^n}$ are generated by $-u,e_1,\cdots,e_n$.

For $B_{n,k}$, let $c:=n-k$, and let $N$ have basis $e_1,\dots,e_k,f_1,\dots,f_c$. The one-dimensional cones of $\Sigma_{B_{n,k}}$ are generated by
\[
e_1,\dots,e_k,\quad f_1,\dots,f_c,\quad
v_0:=-(e_1+\cdots+e_k+f_1+\cdots+f_c),\quad
w:=f_1+\cdots+f_c.
\]

For $C_{s,r_1,r_2}$, set $t=r_1-1$. The fan $\Sigma_{C_{s,r_1,r_2}}$ is in $N_{\RR}\cong \RR^{s+t+r_2}$. Let $e_1,\cdots,e_s, f_1,\cdots,f_t, g_1,\cdots, g_{r_2}$ be a basis of $N$, and
then the one-dimensional cones of $\Sigma_{C_{s,r_1,r_2}}$ are generated by
\[e_1,\cdots, e_s,\ f_1,\cdots, f_t, g_1,\cdots g_{r_2},\]
\[u_0:=-(\sum_{a=1}^t f_a+\sum_{j=1}^{r_2} g_j),\quad v_0:=\sum_{j=1}^{r_2}g_j-\sum_{i=1}^s e_i.\]

By the definition of the uniform unimodularity condition, it suffices to verify that every square submatrix, whose columns are selected from the listed primitive generators in the cases $A_{n,k}$, $B_{n,k}$, and $C_{s,r_1,r_2}$, has determinant equal to $0$, $1$, or $-1$. Therefore, these three classes of toric varieties satisfy the uniform unimodularity condition.

We now verify the Thomsen stratification intersection-number condition for these three families, using the notation for the fans and their primitive generators introduced above. Since they already satisfy the uniform unimodularity condition, Lemma~\ref{le: A(X)=1} shows that it is enough to prove that for any $d$-dimensional stratum $\sigma\in S_X$ and any $T$-invariant curve $C$ on $X$, we have
\begin{equation}
    \label{eq: aim}
    d_{\sigma}\cdot C\le d.
\end{equation}
\begin{enumerate}
    \item Let $X=A_{n,k}$ and let $(x_1,\cdots,x_n)$ be coordinates on $T_M$, where $x_i\in[0,1)$. Then the stratification $S_X$ of $T_M$ is induced by the toric hyperplanes
\[
x_i=0\quad (1\le i\le n),\quad x_1+x_2+\cdots+x_n\in\mathbb{Z}.
\]

Fix a stratum $\sigma$, and set
\[
I=\{i\mid x_i>0\text{ on }\sigma\},
\quad
p=|I|,
\quad
S_I=\sum_{i\in I}x_i.
\]
If $p=0$, then $\mathcal{O}_X(d_\sigma)=\mathcal{O}_X$, and the conclusion \eqref{eq: aim} is immediate. Thus we may assume $p\ge 1$. We divide the strata of $S_X$ into the following two types:
\begin{itemize}
    \item \textit{The open case:} if $\ell<S_I<\ell+1$ for some integer $0\le \ell\le p-1$, then $\dim(\sigma)=p$, and
    \[
    d_\sigma=\sum_{i\in I}D_{e_i}+(\ell+1)D_u-\ell D_{-u}.
    \]

    \item \textit{The wall case: }if $S_I=\ell$ for some integer $1\le \ell\le p-1$, then $\dim(\sigma)=p-1$, and
    \[
    d_\sigma=\sum_{i\in I}D_{e_i}+\ell D_u-\ell D_{-u}.
    \]
\end{itemize}
Set $m:=p-\ell$. In both cases, $1\le m\le \dim(\sigma)$. Therefore, it suffices to show that for every $T$-invariant curve $C$,
\[
d_\sigma\cdot C\le m.
\]

Let $\pi: X\to\mathbb{P}^n$, and set $H=\pi^*\mathcal{O}_{\mathbb{P}^n}(1)$.
Let $E_0,\dots,E_{k-1}$ be the exceptional divisors of the blow-up. The relations between the divisors corresponding to the one-dimensional cones and $H,E_i$ are
\[
D_u=E_0,
\quad
D_{-e_i}=E_i\quad (1\le i\le k-1),
\]
\[
D_{-u}\sim H-\sum_{a=1}^{k-1}E_a,
\quad
D_{e_i}\sim H-E_0-\sum_{\substack{1\le a\le k-1\\ a\ne i}}E_a,
\]
where if $i\ge k$, the expression for $D_{e_i}$ means that the sum is taken over all $a=1,\dots,k-1$.
Substituting these expressions into the formula for $d_\sigma$, we obtain
\[
d_\sigma\sim mH-\sum_{a=0}^{k-1}\lambda_aE_a,
\]
where each $\lambda_a\in\{m-1,m\}$.

We now classify the $T$-invariant curves on $A_{n,k}$:
\begin{enumerate}
    \item If a $T$-invariant curve is contained in some exceptional divisor  $E_a$, denote it by $l_a$, then
    \[
    H\cdot l_a=0,
    \quad
    E_a\cdot l_a=-1,
    \quad
    E_b\cdot l_a=0\quad (b\ne a).
    \]
    Hence
    \[
    d_\sigma\cdot l_a=\lambda_a\le m.
    \]

    \item Let $l_{ab}$ be the curve on $A_{n,k}$ obtained as the strict transform of the line in $\mathbb{P}^n$ joining the fixed points $p_a$ and $p_b$. Let
    \[
    T_{ab}\subseteq \{a,b\}
    \]
    denote the set of indices such that $p_a$ or $p_b$ lies in the center of the blow-up. Then $l_{ab}$ satisfies
    \[
    H\cdot l_{ab}=1,
    \quad
    E_t\cdot l_{ab}=
    \begin{cases}
    1 & t\in T_{ab},\\
    0 & t\notin T_{ab}.
    \end{cases}
    \]
    Therefore,
    \[
    d_\sigma\cdot l_{ab}
    =
    m-\sum_{t\in T_{ab}}\lambda_t
    \le m.
    \]
\end{enumerate}

In conclusion, for every stratum $\sigma$ and every $T$-invariant curve $C$, we have
\[
d_\sigma\cdot C\le m\le \dim(\sigma).
\]
This proves \eqref{eq: aim} for $X=A_{n,k}$.

\item Let $X=B_{n,k}$ and let $(z_1,\dots,z_k,x_1,\dots,x_c)$ be coordinates on $T_M$, represented by variables in $[0,1)$. The strata of $S_X$ are determined by the toric hyperplanes
\[
z_j=0,\ x_i=0,\ 
v:=z_1+\cdots+z_k+x_1+\cdots+x_c\in\mathbb{Z},\ 
u:=x_1+\cdots+x_c\in\mathbb{Z}.
\]

Fix a stratum $\sigma$, and define
\[
I=\{i\mid x_i>0\text{ on }\sigma\},
\quad
J=\{j\mid z_j>0\text{ on }\sigma\}.
\]
Write $p=|I|$, $q=|J|$, and set $\alpha=\lceil u\rceil$, $\beta=\lfloor v\rfloor$. Then
\[
d_\sigma=\sum_{i\in I}D_{f_i}+\sum_{j\in J}D_{e_j}+\alpha D_w-\beta D_{v_0}.
\]
Let $H$ be the pullback of a hyperplane in $\mathbb{P}^n$, and let $E$ be the exceptional divisor of the blow-up. Then
\[
D_{f_i}\sim H-E,
\quad
D_{e_j}\sim H,
\quad
D_{v_0}\sim H,
\quad
D_w=E.
\]
Hence
\begin{equation}
    \label{eq: d_sigma}
    d_\sigma\sim (p+q-\beta)H-(p-\alpha)E.
\end{equation}

The $T$-invariant curves on $B_{n,k}$ are represented by three numerical types $C_0$, $M$ and $F$ such that:
\[
    (H,E)\cdot C_0=(1,0),\quad
    (H,E)\cdot M=(1,1),\quad
    (H,E)\cdot F=(0,-1).
\]

Substituting these intersection numbers into \eqref{eq: d_sigma}, we obtain
\[
d_\sigma\cdot C_0=p+q-\beta,\quad
d_\sigma\cdot M=q+\alpha-\beta,\quad
d_\sigma\cdot F=p-\alpha.
\]

Since $0\le \alpha\le p$ and $\beta\le \alpha+q$, all three intersection numbers are bounded above by $p+q-\beta$. It remains to compare this number with $\dim(\sigma)$. If only one of $p,q$ is nonzero, there is only one possible wall, and the wall case subtracts one from the dimension while forcing $\beta\ge 1$. If both $p$ and $q$ are nonzero, the two walls $u\in\mathbb{Z}$ and $v\in\mathbb{Z}$ are independent; moreover, if both occur, then $u\ge 1$ and $v-u=\sum_{j\in J}z_j$ is a positive integer, so $v\ge 2$. In all cases,
    \[
    p+q-\beta\le \dim(\sigma),
    \]
    and hence $d_\sigma\cdot C\le \dim(\sigma)$.
This proves \eqref{eq: aim} for $X=B_{n,k}$.

\item Let $X=C_{s,r_1,r_2}$ and let $(x_1,\dots,x_s,u_1,\dots,u_t,y_1,\dots,y_{r_2})$ be coordinates on $T_M$, where $x_i,u_a,y_j\in[0,1)$. Then the stratification $S_X$ of $T_M$ is induced by toric hyperplanes
\[
x_i=0\quad (1\le i\le s),\quad
u_a=0\quad (1\le a\le t),\quad
y_j=0\quad (1\le j\le r_2),
\]
together with
\[
F:=\sum_{a=1}^t u_a+\sum_{j=1}^{r_2}y_j\in\mathbb{Z},
\quad
G:=\sum_{j=1}^{r_2}y_j-\sum_{i=1}^s x_i\in\mathbb{Z}.
\]

Fix a stratum $\sigma$. Define
\[
I=\{i\mid x_i>0\text{ on }\sigma\},\ 
A=\{a\mid u_a>0\text{ on }\sigma\},\ 
J=\{j\mid y_j>0\text{ on }\sigma\},
\]
and set
\[
p:=|I|,
\quad
r:=|A|,
\quad
q:=|J|,
\quad
f:=\lfloor F\rfloor,
\quad
g:=\lceil G\rceil.
\]
With this notation, we have
\[
d_\sigma
=
\sum_{i\in I}D_{e_i}
+\sum_{a\in A}D_{f_a}
+\sum_{j\in J}D_{g_j}
+gD_{v_0}
-fD_{u_0}.
\]

Let $\pi:X\to\mathbb{P}^s$ be the projection, set $H=\pi^*\mathcal{O}_{\mathbb{P}^s}(1)$, and let $\xi$ be the tautological class on $X$. We have
\[
D_{e_i}\sim D_{v_0}\sim H,
\quad
D_{f_a}\sim D_{u_0}\sim \xi,
\quad
D_{g_j}\sim \xi-H.
\]
Therefore
\begin{equation}
    \label{eq: d_sigma in type C}
    d_\sigma\sim \alpha H+\beta \xi,
    \quad
    \alpha:=p-q+g,\quad
    \beta:=r+q-f.
\end{equation}

Next, we classify the $T$-invariant curves on $X$. On a toric variety, every $T$-invariant curve corresponds to a wall relation. For $X=C_{s,r_1,r_2}$, the wall relations are precisely the following:
\[
\sum_{a=1}^t f_a+\sum_{j=1}^{r_2} g_j+u_0=0, \quad 
\sum_{i=1}^s e_i+v_0-\sum_{j=1}^{r_2}g_j=0,\quad
\sum_{i=1}^s e_i+v_0+\sum_{a=1}^t f_a+u_0=0.
\]
Let $\ell$, $\ell_0$, and $\ell_1$ denote the corresponding $T$-invariant curves. Then
\[
(H,\xi)\cdot \ell=(0,1), \quad
(H,\xi)\cdot \ell_0=(1,0), \quad
(H,\xi)\cdot \ell_1=(1,1).
\]
Substituting the intersection numbers into \eqref{eq: d_sigma in type C}, we obtain
\[
d_\sigma\cdot \ell=\beta,
\quad
d_\sigma\cdot \ell_0=\alpha,
\quad
d_\sigma\cdot \ell_1=\alpha+\beta.
\]
It remains to compare $\alpha$, $\beta$, and $\alpha+\beta$ with $\dim(\sigma)$.
To this end, let $\delta_F,\delta_G\in\{0,1\}$ be defined by $\delta_F=1$ if $F\in\mathbb Z$ and $\delta_F=0$ otherwise; similarly for $\delta_G$. 
\begin{enumerate}
    \item If at least two of $p,q,r$ are nonzero, then
    \[
    \dim(\sigma)=p+q+r-\delta_F-\delta_G.
    \]
    We have 
    \[G<q\Longrightarrow g-q\le -\delta_G,\]
    and since $q+r-\delta_F\ge 0$, it follows that
    \[
    \alpha=p-q+g\le p+r+q-\delta_F-\delta_G.
    \]
    Next, to prove that $\alpha+\beta\le \dim(\sigma)$, it suffices to show that
    \[
    g-f\le q-\delta_F-\delta_G.
    \]
    Indeed, since $g\le q-\delta_G$ and $f\ge \delta_F$, we obtain $g-f\le q-\delta_F-\delta_G$.

    Finally, to prove $\beta\le \dim(\sigma)$, it suffices to show that
    \[
    -f\le p-\delta_F-\delta_G.
    \]
    Since $-f\le 0$, it remains only to consider the case where
    \[
    p-\delta_F-\delta_G<0.
    \]
    This leaves the following possibilities.
    \begin{enumerate}
        \item\label{itm:b1} If $p=1$, $\delta_F=1$, and $\delta_G=1$, then at least one of $r,q$ is nonzero since our assumption is that at least two among $p,q,r$ are nonzero. Hence $F\in\mathbb{Z}$ and $F\ge 1$, so $-f\le -1=p-\delta_F-\delta_G$.
        \item If $p=0$, $\delta_F=1$, and $\delta_G=0$, then the same argument as in \eqref{itm:b1} yields $-f\le -1=p-\delta_F-\delta_G$.
        \item If $p=0$, $\delta_F=0$, and $\delta_G=1$, then $G=\sum_{j=1}^{r_2} y_j\ge 1$, and hence $F\ge 1$. Therefore $-f\le -1=p-\delta_F-\delta_G$.
        \item If $p=0$, $\delta_F=1$, and $\delta_G=1$, then both $r$ and $q$ are nonzero. Thus
        \[
        G=\sum_{j=1}^{r_2} y_j\in\mathbb{Z}\Longrightarrow G\ge 1,
        \]
        \[
        F=\sum_{a=1}^t u_a+\sum_{j=1}^{r_2} y_j\in\mathbb{Z}\Longrightarrow F\ge 2.
        \]
        Therefore
        \[
        -f\le -2=p-\delta_F-\delta_G.
        \]
    \end{enumerate}
    \item If exactly one of $p,q,r$ is nonzero, we distinguish three cases.
    \begin{enumerate}
        \item Suppose that $p=0,\ r=0$, and $q\ge 1$. In this case $F$ and $G$ contribute the same hyperplane, and hence
        \[
        \dim(\sigma)=q-\delta_F=q-\delta_G.
        \]
        Moreover,
        \[
        \alpha=-q+g,\quad \beta=q-f,\quad F=G,\quad f=\lfloor F\rfloor,\quad g=\lceil G\rceil.
        \]
        Therefore,
        \[
        G<q\Longrightarrow g\le q\Longrightarrow \alpha\le 0\le \dim(\sigma),
        \]
        \[
        f\ge \delta_F\Longrightarrow \beta\le \dim(\sigma),
        \]
        \[
        \alpha+\beta=g-f=1-\delta_F\le q-\delta_F=\dim(\sigma).
        \]
        \item Suppose that $p=0,\ r\ge 1$, and $q=0$.
        In this case
        \[
        G=g=0,\quad \alpha=0,\quad \beta=r-f,\quad \dim(\sigma)=r-\delta_F.
        \]
        Hence
        \[
        f\ge \delta_F\Longrightarrow \beta=\alpha+\beta\le \dim(\sigma).
        \]

        \item Suppose that $p\ge 1,\ r=0$, and $q=0$.
        In this case
        \[
        F=f=0,\quad G=-\sum_{i=1}^s x_i,\quad \alpha=p+g,\quad \beta=0,\quad \dim(\sigma)=p-\delta_G.
        \]
        Hence
        \[
        g=\lceil G\rceil\le -\delta_G\Longrightarrow \alpha=\alpha+\beta\le \dim(\sigma).
        \]
   \end{enumerate}
\end{enumerate}

Combining the above cases, we conclude that for every stratum $\sigma$ and every $T$-invariant curve $C$, we have
\[
d_\sigma\cdot C\le \dim(\sigma).
\]
Hence this proves \eqref{eq: aim} for $X=C_{s,r_1,r_2}$.

\end{enumerate}

\end{proof}

\section{Property $N_p$}
After introducing the two  conditions on toric varieties and verifying them in examples, we turn to the proof of our main result, Theorem \ref{thm: main theorem} on Property $N_p$ for line bundles on toric varieties.

We prove this theorem using Green's criterion (see \cite{Green1984b} for the original statement, and also \cite{Inamdar97, BangereLacini2025} for a careful treatment). We first introduce some notation. Let $2 \le s \le r$ and $1 \le i < j \le r$ be positive integers, and let $X$ be a smooth projective variety of dimension $n$. Denote by
\[
X(r)=X\times \cdots \times X
\]
the product of $r$ copies of $X$. Let
\[
\mathrm{pr}_i:X(r)\to X
\]
and
\[
\mathrm{pr}_{i,j}=\mathrm{pr}_i\times \mathrm{pr}_j:X(r)\to X\times X
\]
be the natural projection maps. Let $\Delta_X$ be the diagonal of $X(2)$. We write
\[
\Delta^{i,j}_{X(r)}=\mathrm{pr}_{i,j}^*\Delta_X
\]
for the diagonal subvariety corresponding to the $i$-th and $j$-th factors.

Set
\[
\Delta^1_{X(r)}(s)=\bigcup_{2\le j\le s}\Delta^{1,j}_{X(r)}.
\]
We also define
\[
\mathcal{I}^1_{\Delta_{X(r)}}(s)
=\mathcal{I}_{\Delta^{1,2}_{X(r)}}\otimes \mathcal{I}_{\Delta^{1,3}_{X(r)}}\otimes\cdots\otimes \mathcal{I}_{\Delta^{1,s}_{X(r)}}.
\]
Clearly,
\[
\mathcal{I}^1_{\Delta_{X(2)}}(2)
=\mathcal{I}_{\Delta^{1,2}_{X(2)}}=\mathcal{I}_{\Delta_X},
\]
and we adopt the convention that
\[
\mathcal{I}^1_{\Delta_{X(r)}}(1)=\mathcal{O}_{X(r)}.
\]
It holds that (see \cite[Corollary 3.5]{BangereLacini2025})
\begin{equation}
    \label{eq: ideal sheaf of union of diagonals}
    \mathcal{I}^1_{\Delta_{X(r)}}(s)
\cong
\mathcal{I}_{\Delta^{1,2}_{X(r)}}\cdots \mathcal{I}_{\Delta^{1,s}_{X(r)}}
=
\mathcal{I}_{\Delta^{1,2}_{X(r)}}\cap\cdots\cap \mathcal{I}_{\Delta^{1,s}_{X(r)}}
=
\mathcal{I}_{\Delta^1_{X(r)}(s)}.
\end{equation}

Green's criterion is as follows.
\begin{proposition}
    \label{prop: Green's criterion}
    Let $X$ be a smooth projective variety and $L$ be a basepoint-free line bundle on $X$. Then $L$ satisfies Property $N_p$ provided that
    \begin{equation}
        \label{eq: cohomology vanishing}
        H^1(X(r), \mathcal{I}^1_{\Delta_{X(r)}}(r)\otimes L^{\otimes k}\boxtimes L^{\boxtimes r-1})=0
    \end{equation}
    for all integers $2\le r\le p+2$ and $k\ge 1$.\qed
\end{proposition}

\begin{proof}[Proof of Theorem
    \ref{thm: main theorem}]
    By \cite{GZ22}, $L$ is very ample,  hence  in  particular basepoint-free.
    We shall apply Proposition \ref{prop: Green's criterion}. To this end, we fix integers $2\le r\le p+2$ and $k\ge 1$, and aim to show (\ref{eq: cohomology vanishing}).

    First, in \cite{HanlonHicksLazarev2024}, for every toric subvariety $Y$ of a smooth toric variety $X$, the authors construct an explicit resolution of the structure sheaf $\mathcal{O}_Y$, called HHL resolution. Since $\Delta^{1,j}_{X(r)}$ is a toric subvariety of $X(r)$, we can construct the HHL resolution of $\mathcal{O}_{\Delta^{1,j}_{X(r)}}$ denoted by $C_{\bullet}^{1,j}$. Following the construction of the HHL resolution, we have 
\[ C_{s}^{1,j}=\bigoplus_{\substack{\sigma\in S_X\\ \dim(\sigma)=s}}  \mathcal{O}_X(-d_{\sigma})\boxtimes \mathcal{O}_X\boxtimes \cdots \boxtimes \mathcal{O}_X\boxtimes \mathcal{O}_X(-d_{-\sigma})\boxtimes \mathcal{O}_X\boxtimes \cdots \boxtimes \mathcal{O}_X,\]
where $\mathcal{O}_X(-d_{-\sigma})$ lies at $j$th position. Here, $-\sigma$ denotes the stratum $\{-\theta\mid \theta\in \sigma\}$ and $d_{\sigma}$ is given by \eqref{eq: -d(theta)} and \eqref{eq: def of d_sigma}. Note that strata in $S_X$ have dimension at most $n$, so $C_{s}^{1,j}=0$ for $s>n$ or $s<0$. Since $X$ satisfies the uniform unimodularity condition, $S_X$ has only one $0$-dimensional stratum, which is the trivial point $\{0\}$ by Proposition \ref{prop: equivalent condition for trivial 0 stratum}. Hence, $C_{0}^{1,j}=\mathcal{O}_{X(r)}$ for all $j$. Then $C_{\bullet}^{1,j}$  naturally gives rise to a resolution of $\mathcal{I}_{\Delta^{1,j}_{X(r)}}$ by deleting $C_{0}^{1,j}$ and shifting the indices by $-1$. We continue to denote this resolution by $C_{\bullet}^{1,j}$, so
\[ C_{s}^{1,j}=\bigoplus_{\substack{\sigma\in S_X\\ \dim(\sigma)=s+1}}  \mathcal{O}_X(-d_{\sigma})\boxtimes \mathcal{O}_X\boxtimes \cdots \boxtimes \mathcal{O}_X\boxtimes \mathcal{O}_X(-d_{-\sigma})\boxtimes \mathcal{O}_X\boxtimes \cdots \boxtimes \mathcal{O}_X\]
and $C_{s}^{1,j}=0$ for $s>n-1$ or $s<0$.

Now fix a $k\ge 1$. Let 
\[\widetilde{C}_{\bullet}^r=T\paren{\bigotimes_{j=2}^r C_{\bullet}^{1,j} \otimes L^{\otimes k}\boxtimes L^{\boxtimes r-1}}\]
be the total complex, which is a resolution of $\mathcal{I}_{\Delta_{X(r)}}^1(r)\otimes L^{\otimes k}\boxtimes L^{\boxtimes r-1}$ cf.~\cite[Corollary 3.4]{BangereLacini2025}. We have
\[
        \begin{aligned}
            \widetilde{C}^r_s=\bigoplus_{\substack{0\le p_i\le n-1 \\ \sum_{i=1}^{r-1} p_i=s}}\bigoplus_{\substack{\sigma_1,\cdots\sigma_{r-1}\in S_X\\ \dim(\sigma_i)=p_i+1}} & \Big(\mathcal{O}_X(-\sum_{i=1}^{r-1} d_{\sigma_i})\otimes L^k\Big)\\
            & \boxtimes \Big(\mathcal{O}_X(-d_{-\sigma_1})\otimes L\Big) 
    \boxtimes \cdots \boxtimes \Big(\mathcal{O}_X(-d_{-\sigma_{r-1}})\otimes L\Big)
        \end{aligned}
    \]
    To prove 
    \[H^1(X(r), \mathcal{I}^1_{\Delta_{X(r)}}(r)\otimes L^{\otimes k}\boxtimes L^{\boxtimes r-1})=0,\]
    it suffices to show that  
    \begin{equation}
        \label{eq: cohomology vanishing for C^r_s}
        H^{s+1}(X(r), \widetilde{C}^r_s)=0
    \end{equation} 
    for all integers $s\ge 0$.

    In fact, by the K\"unneth formula, we have
    \[
        \begin{aligned}
            H^{s+1}(X(r), \widetilde{C}_s^r)=\bigoplus_{q_1+\cdots +q_r=s+1}\bigoplus_{\substack{0\le p_i\le n-1 \\ \sum_{i=1}^{r-1} p_i=s}}\bigoplus_{\substack{\sigma_1,\cdots,\sigma_{r-1}\in S_X\\ \dim(\sigma_i)=p_i+1}} &\Big(H^{q_1}(\mathcal{O}_X(-\sum_{i=1}^{r-1} d_{\sigma_i})\otimes L^k)\\
        &\otimes H^{q_2}(\mathcal{O}_X(-d_{-\sigma_1})\otimes L)\\
        &\otimes\cdots\otimes H^{q_r}(\mathcal{O}_X(-d_{-\sigma_{r-1}})\otimes L)\Big).
        \end{aligned}
    \]
    By Lemma \ref{le: cohomology vanishing for Thomsen collection}, we have
    \[H^{q_{i+1}}(\mathcal{O}_X(-d_{-\sigma_{i}})\otimes L)=0, \text{ for all } q_{i+1}\ne 0,\ 1\le i\le r-1.\]
Therefore, it remains only to consider the case where $q_1=s+1$ and $q_i=0$ for $2\le i\le r$, namely,
\[
    \begin{aligned}
        H^{s+1}(\widetilde{C}_s^r)=\bigoplus_{\substack{0\le p_i\le n-1 \\ \sum_{i=1}^{r-1} p_i=s}}\bigoplus_{\substack{\sigma_1,\cdots,\sigma_{r-1}\in S_X\\ \dim(\sigma_i)=p_i+1}} &\Big(H^{s+1}(\mathcal{O}_X(-\sum_{i=1}^{r-1} d_{\sigma_i})\otimes L^k)\\
        &\otimes H^{0}(\mathcal{O}_X(-d_{-\sigma_1})\otimes L)\\
        &\otimes\cdots\otimes H^{0}(\mathcal{O}_X(-d_{-\sigma_{r-1}})\otimes L)\Big).
    \end{aligned}
\]
Since $r\le p+2$, the hypothesis $L\cdot C\ge n-1+p$ implies $L\cdot C\ge n+r-3$ for every $T$-invariant curve $C$. Therefore, to prove \eqref{eq: cohomology vanishing for C^r_s}, it suffices to show that, for any $s\ge 0$, any integers $0\le p_1,\cdots,p_{r-1}\le n-1$ satisfying $\sum_{i=1}^{r-1}p_i=s$, and any $\sigma_1,\cdots,\sigma_{r-1}\in S_X$ with $\dim(\sigma_i)=p_i+1$, one has
\begin{equation}
    \label{eq: vanishing in induction}
    H^{s+1}(X,\mathcal{O}_X(-\sum_{i=1}^{r-1} d_{\sigma_i})\otimes L^k)=0.
\end{equation}
Since $X$ is $n$-dimensional, it is enough to consider $0\le s\le n-1$. Note that for each $\sigma_i$, we have $\dim(\sigma_i)\ge 1$. Hence
\[
\sum_{i=1}^{r-2}\dim(\sigma_i)=\sum_{i=1}^{r-1} p_i+r-1-\dim(\sigma_{r-1})=s+r-1-\dim(\sigma_{r-1})\le n+r-3.
\]

Since $X$ satisfies the Thomsen stratification intersection-number condition, using \eqref{eq: simplified intersection number condition}, we have
\[
\mathcal{O}_X\Big(\sum_{i=1}^{r-2} d_{\sigma_i}\Big)\cdot C\le \sum_{i=1}^{r-2} \dim(\sigma_i)\le n+r-3.
\]
Therefore, for any line bundle satisfying $L\cdot C\ge n+r-3$, we obtain that 
\[
\Big(L^k\otimes \mathcal{O}_X(-\sum_{i=1}^{r-2} d_{\sigma_i})\Big)\cdot C\ge 0,
\]
which implies that $L^k\otimes \mathcal{O}_X(-\sum_{i=1}^{r-2} d_{\sigma_i})$ is nef \cite[Theorem 6.3.12]{CoxToric}.
By Lemma~\ref{le: cohomology vanishing for Thomsen collection}, it follows that
\[
H^q\Big(X,\mathcal{O}_X(-d_{\sigma_{r-1}})\otimes L^k\otimes \mathcal{O}_X(-\sum_{i=1}^{r-2} d_{\sigma_i})\Big)=0,\quad q\ge 1,
\]
hence \eqref{eq: vanishing in induction} holds, and so does \eqref{eq: cohomology vanishing for C^r_s}. This completes the proof.
\end{proof}

%Theorem~\ref{thm: main theorem} proves Conjecture~\ref{conj:Song-Zhu} for smooth toric varieties satisfying the uniform unimodularity condition and the Thomsen stratification intersection-number condition.
 We now give a combinatorial version of Theorem~\ref{thm: main theorem}. For toric varieties, the intersection number $L\cdot C$ admits a natural combinatorial interpretation. Given an ample line bundle $L=\mathcal{O}_X(D)$ on a toric variety $X$, recall from \cite{CoxToric} that $L$ corresponds to a lattice polytope $P_D$ in $M_{\mathbb{R}}$. Moreover, the $T$-invariant curves on $X$ are in one-to-one correspondence with the edges of $P_D$; denote by $e$ the edge corresponding to a curve $C$. It is well-known that
\[
L\cdot C=\#(e\cap M)-1.
\]
We call this number the lattice length of the edge $e$ and denote it by $l(e)$. Let $E(P_D)$ denote the set of edges of $P_D$. 
\begin{definition}
    The lattice length of $P_D$ is defined by
    \[
    L(P_D)=\min\{l(e)\mid e\in E(P_D)\}.
    \]
\end{definition}

Thus $L(P_D)$ is the minimum of the intersection numbers of the line bundle $L$ with all $T$-invariant curves.

% \begin{figure}
%     \centering 
%     \includegraphics[width=0.35\textwidth]{image/lattice length of polytope.jpg}
%     \caption{An illustration of a polytope of lattice length $2$}
% \end{figure}

\begin{corollary}
    Let $P\subseteq M_{\mathbb{R}}$ be a smooth $n$-dimensional lattice polytope. If the associated toric variety $X_P$ satisfies the uniform unimodularity condition and the Thomsen stratification intersection-number condition, and if $L(P)\ge n-1+p$, then $P$ satisfies Property $N_p$.
\end{corollary}

\section{Open Questions}
We conclude by discussing two directions suggested by the preceding results.

The first direction concerns the uniform unimodularity condition. This condition is substantially stronger than smoothness for toric varieties. By the results of \cite{Heller1957} and \cite{OxleyWalsh2022}, if $\Sigma$ is a fan satisfying the uniform unimodularity condition, then the number of one-dimensional cones in $\Sigma$ satisfies
\[
|\Sigma(1)|\le n(n+1),
\]
where $n=\operatorname{rank}(N)$, and this bound is sharp. If one identifies $u$ and $-u$ as the same direction, then $\Sigma(1)$ contains at most $n(n+1)/2$ distinct directions. On the other hand, for a toric variety one has
\[
\rho(X_{\Sigma})=|\Sigma(1)|-n,
\]
where $\rho(X_{\Sigma})$ denotes the Picard number of $X_{\Sigma}$. 

Therefore, for toric varieties satisfying the uniform unimodularity condition, the Picard number is bounded above. This naturally leads to the following question:

\begin{problem}
    Can one give a complete classification of proper toric varieties satisfying the uniform unimodularity condition?
\end{problem}

The second direction concerns the Thomsen stratification intersection-number condition. Proposition~\ref{prop: example} shows that this condition holds for the basic families $A_{n,k}$, $B_{n,k}$ and $C_{s,r_1,r_2}$. However, the verification relies on direct, case-by-case computations. At present, no conceptual or general method is known for establishing the condition for an arbitrary smooth toric variety. This motivates the following conjecture.

\begin{conjecture}
    For any smooth proper toric variety, the Thomsen stratification intersection-number condition holds. 
\end{conjecture}

\bibliographystyle{alpha}
\bibliography{reference}

@article{ballard2025kingsconjecturecoxcategory,
title={King's Conjecture and the {C}ox category}, 
author={Matthew R. Ballard and Christine Berkesch and Michael K. Brown and Lauren Cranton Heller and Daniel Erman and David Favero and Sheel Ganatra and Andrew Hanlon and Jesse Huang},
journal={arXiv preprint},
year={2025},
volume={arXiv:2501.00130},
archivePrefix={arXiv},
primaryClass={math.AG},
url={https://arxiv.org/abs/2501.00130}, 
}

@article{BangereLacini2025,
author = {Purnaprajna Bangere and Justin Lacini},
title = {{Syzygies of adjoint linear series on projective varieties}},
volume = {174},
journal = {Duke Mathematical Journal},
number = {3},
publisher = {Duke University Press},
pages = {473--499},
year = {2025},
doi = {10.1215/00127094-2024-0035},
}

@inproceedings{Bondal06,
  title={Derived categories of toric varieties},
  author={Bondal, Alexey},
  booktitle={Convex and Algebraic geometry, Oberwolfach conference reports, EMS Publishing House},
  volume={3},
  pages={284--286},
  year={2006}
}

@book {CoxToric,
    AUTHOR = {Cox, David A. and Little, John B. and Schenck, Henry K.},
     TITLE = {Toric varieties},
    SERIES = {Graduate Studies in Mathematics},
    VOLUME = {124},
 PUBLISHER = {American Mathematical Society, Providence, RI},
      YEAR = {2011},
      ISBN = {978-0-8218-4819-7},
   MRCLASS = {14M25 (05A15 05E45 52B12)},
  MRNUMBER = {2810322},
MRREVIEWER = {Ivan\ Arzhantsev},
       DOI = {10.1090/gsm/124},
       URL = {https://doi.org/10.1090/gsm/124},
}

@book {FultonIntroToToric,
    AUTHOR = {Fulton, William},
     TITLE = {Introduction to toric varieties},
    SERIES = {Annals of Mathematics Studies},
    VOLUME = {131},
 PUBLISHER = {Princeton University Press, Princeton, NJ},
      YEAR = {1993},
     PAGES = {xii+157},
      ISBN = {0-691-00049-2},
   MRCLASS = {14M25 (14-02 14J30)},
  MRNUMBER = {1234037},
MRREVIEWER = {T.\ Oda},
       DOI = {10.1515/9781400882526},
       URL = {https://doi.org/10.1515/9781400882526},
}

@article{GZ22,
author = {José L. González and Zhixian Zhu},
title = {Generation of jets and {F}ujita's jet ampleness conjecture on toric varieties},
journal = {Journal of Pure and Applied Algebra},
volume = {226},
number = {4},
pages = {106873},
year = {2022}
}

@article {HanlonHicksLazarev2024,
    AUTHOR = {Hanlon, Andrew and Hicks, Jeff and Lazarev, Oleg},
     TITLE = {Resolutions of toric subvarieties by line bundles and
              applications},
   JOURNAL = {Forum Math. Pi},
  FJOURNAL = {Forum of Mathematics. Pi},
    VOLUME = {12},
      YEAR = {2024},
     PAGES = {Paper No. e24, 58}
}

@article{Heller1957,
author = {Heller, Isidore},
title = {On linear systems with integral valued solutions},
volume = {7},
journal = {Pacific Journal of Mathematics},
number = {3},
pages = {1351--1364},
year = {1957},
}

@article {HeringSchenckSmith2006,
    AUTHOR = {Hering, Milena and Schenck, Hal and Smith, Gregory G.},
     TITLE = {Syzygies, multigraded regularity and toric varieties},
  JOURNAL = {Compositio Mathematica},
    VOLUME = {142},
      YEAR = {2006},
    NUMBER = {6},
     PAGES = {1499--1506},
      ISSN = {0010-437X,1570-5846},
   MRCLASS = {13D02 (14M25 52B20)},
  MRNUMBER = {2278757},
MRREVIEWER = {P.\ Schenzel},
       DOI = {10.1112/S0010437X0600251X},
       URL = {https://doi.org/10.1112/S0010437X0600251X},
}

@article {Koelman1993,
    AUTHOR = {Koelman, Robert Jan},
     TITLE = {A criterion for the ideal of a projectively embedded toric surface to be generated by quadrics},
  JOURNAL = {Beitr\"age zur Algebra und Geometrie. Contributions to Algebra
              and Geometry},
    VOLUME = {34},
      YEAR = {1993},
    NUMBER = {1},
     PAGES = {57--62},
      ISSN = {0138-4821},
   MRCLASS = {14M25 (14J25)},
  MRNUMBER = {1239278},
MRREVIEWER = {T.\ Oda},
}

@article {Schenck2004,
    AUTHOR = {Schenck, Hal},
     TITLE = {Lattice polygons and {G}reen's theorem},
  JOURNAL = {Proceedings of the American Mathematical Society},
    VOLUME = {132},
      YEAR = {2004},
    NUMBER = {12},
     PAGES = {3509--3512},
      ISSN = {0002-9939,1088-6826},
   MRCLASS = {52B35 (14M25)},
  MRNUMBER = {2084071},
       DOI = {10.1090/S0002-9939-04-07523-9},
       URL = {https://doi.org/10.1090/S0002-9939-04-07523-9}
}

@article {OgataNakagawa2002,
    AUTHOR = {Ogata, Shoetsu and Nakagawa, Katsuyoshi},
     TITLE = {On generators of ideals defining projective toric varieties},
  JOURNAL = {Manuscripta Mathematica},
    VOLUME = {108},
      YEAR = {2002},
    NUMBER = {1},
     PAGES = {33--42},
      ISSN = {0025-2611,1432-1785},
   MRCLASS = {14M25 (52B20)},
  MRNUMBER = {1912946},
MRREVIEWER = {Henry\ K.\ Schenck},
       DOI = {10.1007/s002290200252},
       URL = {https://doi.org/10.1007/s002290200252},
}

@article {Gubeladze2012,
    AUTHOR = {Gubeladze, Joseph},
     TITLE = {Convex normality of rational polytopes with long edges},
  JOURNAL = {Advances in Mathematics},
    VOLUME = {230},
      YEAR = {2012},
    NUMBER = {1},
     PAGES = {372--389},
      ISSN = {0001-8708,1090-2082},
   MRCLASS = {52B20 (14M25)},
  MRNUMBER = {2900547},
MRREVIEWER = {Milena\ S.\ Hering},
       DOI = {10.1016/j.aim.2011.12.003},
       URL = {https://doi.org/10.1016/j.aim.2011.12.003},
}

@article {Green1984,
    AUTHOR = {Green, Mark L.},
     TITLE = {Koszul cohomology and the geometry of projective varieties. {I}},
   JOURNAL = {Journal of Differential Geometry},
    VOLUME = {19},
      YEAR = {1984},
    NUMBER = {1},
     PAGES = {125--171},
      ISSN = {0022-040X,1945-743X},
   MRCLASS = {14F05 (14B12)},
  MRNUMBER = {739785},
MRREVIEWER = {G.\ Horrocks},
}

@article{Green1984b,
     author = {Green, Mark L.},
     title = {Koszul cohomology and the geometry of projective varieties. {II}},
     JOURNAL = {Journal of Differential Geometry},
     volume = {19},
     number = {1},
     year = {1984},
     pages = {279--289}
}

@article{Inamdar97,
  title={On syzygies of projective varieties},
  author={Inamdar, S. P.},
  journal={Pacific Journal of Mathematics},
  volume={177},
  number={1},
  pages={71--76},
  year={1997},
  publisher={Mathematical Sciences Publishers}
}

@article {EinLazarsfeld1993arbitrarydim,
    AUTHOR = {Ein, Lawrence and Lazarsfeld, Robert},
     TITLE = {Syzygies and {K}oszul cohomology of smooth projective
              varieties of arbitrary dimension},
  JOURNAL = {Inventiones Mathematicae},
    VOLUME = {111},
      YEAR = {1993},
    NUMBER = {1},
     PAGES = {51--67},
      ISSN = {0020-9910,1432-1297},
   MRCLASS = {13D02 (14F17 14J60)},
  MRNUMBER = {1193597},
MRREVIEWER = {Gary\ P.\ Kennedy},
       DOI = {10.1007/BF01231279},
}

@article {Koelman1993Generators,
    AUTHOR = {Koelman, Robert Jan},
     TITLE = {Generators for the ideal of a projectively embedded toric
              surface},
  JOURNAL = {The Tohoku Mathematical Journal. Second Series},
    VOLUME = {45},
      YEAR = {1993},
    NUMBER = {3},
     PAGES = {385--392},
      ISSN = {0040-8735,2186-585X},
   MRCLASS = {14M25 (14J25)},
  MRNUMBER = {1231563},
MRREVIEWER = {T.\ Oda},
       DOI = {10.2748/tmj/1178225891},
       URL = {https://doi.org/10.2748/tmj/1178225891},
}

@article{OxleyWalsh2022,
author = {Oxley, James and Walsh, Zach},
title = {2-Modular Matrices},
journal = {SIAM Journal on Discrete Mathematics},
volume = {36},
number = {2},
pages = {1231-1248},
year = {2022},
doi = {10.1137/21M1419131},
URL = {  https://doi.org/10.1137/21M1419131}
}

@article{SWZ24,
      title={On covering simplices by dilations in dimensions 3 and 4}, 
      author={Song,Lei and Wen, Huanqi and Zhu, Zhixian},
      journal={arXiv preprint arXiv:2404.02495},
      year={2024}
}

@article{SWZ26,
      title={The integral closedness of lattice simplices with large lattice length}, 
      author={Lei Song and Huanqi Wen and Zhixian Zhu},
      year={2026},
      JOURNAL = {arXiv preprint},
      volume={arXiv:2606.16348} 
}

\bigskip
\noindent\small{\textsc{School of Mathematics, Sun Yat-sen University\\
W. 135 Xingang Rd., Guangzhou, Guangdong 510275, P.R.~China}\\
\emph{E-mail address}:  \texttt{songlei3@mail.sysu.edu.cn}

\bigskip
\noindent\small{\textsc{School of Mathematics, Sun Yat-sen University\\
W. 135 Xingang Rd., Guangzhou, Guangdong 510275, P.R.~China}\\
\emph{E-mail address}:  \texttt{wenhq7@mail2.sysu.edu.cn}

\end{document}